\documentclass[11pt,a4paper,reqno]{amsart}
\usepackage{amsmath}
\usepackage{amsfonts}
\usepackage{amssymb}
\usepackage{amscd}
\usepackage{url}
\usepackage[pdftex,bookmarks=true]{hyperref}

\newcommand\rank{{\operatorname{rank}}}

\newcommand\R{{\mathbf{R}}}

\renewcommand\P{{\mathbf{P}}}

\newcommand\eps{\varepsilon}

\newcommand\Ba{{\mathbf a}}

\newcommand\Be{{\mathbf e}}

\newcommand\Bh{{\mathbf h}}

\newcommand\Br{{\mathbf r}}

\newcommand\Bu{{\mathbf u}}
\newcommand\Bv{{\mathbf v}}

\newcommand\Bx{{\mathbf x}}

\newcommand\BG{{\mathbf G}}
\newcommand\BH{{\mathbf H}}

\newcommand\ep{\epsilon}

\theoremstyle{plain}
  \newtheorem{theorem}[subsection]{Theorem}
  \newtheorem{conjecture}[subsection]{Conjecture}

  \newtheorem{lemma}[subsection]{Lemma}
  \newtheorem{corollary}[subsection]{Corollary}

\theoremstyle{remark}

  \newtheorem{claim}[subsection]{Claim}

\theoremstyle{definition}

\begin{document}

\title[Random matrices of equal row-sums]{On the singularity of random combinatorial matrices}

\author{}
\address{}
\email{}
\subjclass[2000]{}

\maketitle

\begin{abstract}
It is shown that a random $(0,1)$ matrix whose rows are independent random vectors of exactly $n/2$ zero components is non-singular with probability $1-O(n^{-C})$ for any $C>0$. The proof uses a non-standard inverse-type Littlewood-Offord result. 
\end{abstract}

\section{Introduction}\label{section:introduction}

Let $A_n$ denote a random $n$ by $n$ matrix, whose entries are iid Bernoulli random variables. A classical result of Koml\'os \cite{B,K} shows 

$$\P(A_n  \mbox{ is singular })=O(n^{-1/2}).$$ 

By considering the event that two rows or two columns of $A_n$ are equal (up to a sign), it is clear that 

$$\P(A_n  \mbox{ is singular }) \ge (1+o(1))n^2 2^{1-n}.$$ 

It has been conjectured by many researchers that in fact this bound is best possible.

\begin{conjecture}\label{conjecture:1}

$$\P(A_n  \mbox{ is singular })=(\frac{1}{2}+o(1))^n.$$

\end{conjecture}

In a breakthrough paper, Kahn, Koml\'os and Szemer\'edi \cite{KKSz} substantially improved the upper bound to 

$$\P(A_n  \mbox{ is singular })= O(.999^n).$$ 

Another significant improvement is due to Tao and Vu \cite{TVsing} to obtain the bound  

$$\P(A_n  \mbox{ is singular }) = O((3/4)^n).$$ 

The most recent record is due to Bourgain, Vu and Wood \cite{BVW}, who improved it to $O((1/\sqrt{2})^n)$.

All the proofs use a result of Erd\H{o}s and Littlewood-Offord in one or another way. We will discuss this result in details in Section \ref{section:method}. 

Another popular model of random matrices is that of random symmetric matrices. This is one of the simplest models that has non-trivial correlations between the matrix entries. Let $M_n$ denote a random symmetric $n$ by $n$ matrix, whose upper diagonal entries are iid Bernoulli random variables. Despite its obvious similarity to the non-symmetric case $A_n$, less is known concerning the singularity bound for $M_n$.
 
The question to determine whether the singular probability of $M_n$ tends to zero together with $n$ was first posed by Weiss in the early nineties. This question had been open until a recent breakthrough paper by Costello, Tao and Vu \cite{CTV}, who showed $\P(M_n$ is singular$ )= n^{-1/8+o(1)}$. Although the bound $O(n^{-1/8+o(1)})$ can be improved further by applying the more recent inequalities from \cite{C}, it seems that their method cannot give any bound better than $n^{-1/2+o(1)}$. In \cite{Ng-Duke}, the author of this note was able to improve the bound to $O(n^{-C})$, for any $C$. The best current bound is due to Vershynin \cite{V} who shows  $\P(M_n$ is singular$ )= \exp(-n^c)$, for some small positive constant $c$. 

The common feature of the two random matrix ensembles $A_n$ and $M_n$ is that the entries in each row or column are independent. Our main focus is on a simple model where this is not the case.   

Let $n$ be an even number, and let $Q_n$ be a random $(0,1)$ matrix whose rows are independent vector of exactly $n/2$ zero components. This matrix model resembles that of $A_n$ where the entries are iid Bernoulli random variables which take value $0$ and $1$ with probability $1/2$. One can also view $Q_n$ as a non-symmetric version of the adjacency matrix $R_{n,n/2}$ of a random $n/2$-regular graph. 

Estimating the second largest eigenvalue of the adjacency matrices $R_{n,d}$ of random $d$-regular graphs is a well-known problem in combinatorics and theoretical computer science. However, much less is known about the singularity of such random matrix model. It has been conjectured by Vu that 

\vskip .2in

\begin{conjecture}\cite{V} Assume that $d\ge 3$ is a number which may depend on $n$, then $R_{n,d}$, the adjacency matrix of a $d$-regular graph, is non-singular almost surely.
\end{conjecture}

It has been shown (for instance in \cite{TVcir}) that the singularity probability plays a crucial role in proving certain local limit laws involving matrix parameters. The main difficulties of the above conjecture come from two primary constraints: the matrix must be symmetric, and each row vector must have exactly $n-d$ zero entries. Our main goal is to relax the symmetry constraint, working with matrices whose rows are independent $(0,1)$ vectors of exactly $n-d$ zero entries. Although our method may work for a wide range of $d$, we will focus on the case $d=n/2$ only, showing that such random matrix is non-singular almost surely.

\begin{theorem}[Main result]\label{theorem:main} 
For any $C>0$ we have 

$$\P(Q_n \mbox{ is singular }) = O(n^{-C}),$$ 

where the implied constant depends on $C$.
\end{theorem}

By considering the event that two rows of $Q_n$ are equal, it is clear that $Q_n$ is singular with probability at least $(\frac{1}{2}+o(1))^n$. We conjecture that this is also the right bound.

\begin{conjecture}
$$\P(Q_n \mbox{ is singular }) = (\frac{1}{2}+o(1))^n.$$ 
\end{conjecture}

The paper is organized as follows. In the next section we discuss the result of Erd\H{o}s and Littlewood-Offord and its variants in details. After providing further necessary ingredients in Section \ref{section:lemmas}, we complete the proof of Theorem \ref{theorem:main} in Section \ref{section:proof}.

{\bf Notation.} 

For an event $A$, we use the subscript $\P_{\Bx}(A)$ to emphasize that the probability under consideration is taking according to the random vector $\Bx$.

For $1\le s \le n$, we denote by $\Be_s$ the unit vector $(0,\dots,0,1,0,\dots,0)$, where all but the $s$-th component are zero.

\section{The main tools}\label{section:method}

Let $x_i, i=1,\dots, n$ be iid Bernoulli random variables, taking values $\pm 1$ with probability $1/2$. Given a set $A$ of $n$ real number  $a_1, \dots, a_n$, the \emph{concentration probability} of $A$ is defined to be 

$$\rho(A) := \sup_{a} \P\Big(\sum_{i=1}^n a_ix_i =a\Big).$$

Motivated by their study of random polynomials, in the 1940s Littlewood and Offord \cite{LO} raised the question of bounding $\rho(A)$. They showed that if the $a_i$ are nonzero then $\rho(A)=O(n^{-1/2}\log n)$. Shortly after the Littlewood-Offord paper, Erd\H{o}s \cite{E} gave a combinatorial proof of the refinement $\rho(A)\le \binom{n}{n/2}2^{-n} $. 

Since the pioneer results of Erd\H{o}s and Littlewood and Offord, there has been an impressive wave of reasearch to improve the inequality by imposing new assumptions on the $a_i$'s. These improvements include the works of Erd\H{o}s and Moser \cite{EM}, Hal\'asz \cite{H}, Katona\cite{Kat}, Kleitman \cite{K}, S\'ark\"ozy and Szemer\'edi \cite{SSz}, and Stanley \cite{Stan}. 

Recently, Tao and Vu have brought a different view to the problem. Instead of following the classical settings, they tried to find the underlying reason as to why $\rho(A)$ is large, say $\rho(A)\ge n^{-C}$ for some $C>0$. This created a new direction called {\it inverse Littlewood-Offord problem}.

Note that the set  $A$ has $2^n$ subsums, and $\rho({A})\ge n^{-C}$ means that at least $2^n n^{-C}$ among these take the same
value. This observation suggests that  the set should have a very rich additive structure. To determine this structure, let us recall an important concept in Additive Combinatorics, \emph{generalized arithmetic progressions} (GAPs). 

A subset $P$ of $\R$ is a \emph{GAP of rank $r$} if it can be expressed as in the form 

$$P= \Big\{g_0+ m_1g_1 + \dots +m_r g_r| N_i \le m_i \le N_i'\Big\}.$$ 

The numbers $g_i$ are the \emph{generators } of $P$. The numbers $N_i,N_i'$ are the {\it dimensions} of $P$. We say that $P$ is \emph{proper} if every element of $P$ can be written as such a linear combination of the generators in a unique way. If $-N_i=N_i'$ for all $i$ and if $g_0=0$, we say that $P$ is {\it symmetric}.

\vskip .1in

Assume that $P$ is a proper symmetric GAP of rank $r=O(1)$ and size $n^{O(1)}$, and assume that all the elements of $A$ are contained in $P$. Then, by the additive property $|nP|\le n^r|P|$ of $P$, we easily have  $\rho(V) = \Omega(n^{-O(1)})$. 

This example shows that, if the elements of $A$ belong to a symmetric proper GAP with a small rank and small cardinality, then
$\rho(A)$ is indeed very large. A few years ago, Tao and Vu \cite{TVstrong,TVinverse} proved several results showing that this is essentially the only reason. 

{\it Assume that $\rho(A)\ge n^{-C}$ for some $C>0$, then most of the elements of $A$ belong to a symmetric proper GAP of bounded rank $O(1)$ and of small size $n^{O(1)}$.}

Due to the applications they had, the sharpness of the inverse results was not addressed. In a joint work with Vu  we are able to give an optimal version.

\begin{theorem}[Inverse Littlewood-Offord result]\cite[Theorem 2.5]{NgV-Adv}\label{theorem:ILO} 
Let $\eps<1$ and $C$ be positive constants. Assume that

$$\rho (A)  \ge  n^{-C}. $$

Then, for any $n^\ep \le n' \le n$, there exists a proper symmetric GAP $P$ of rank $r=O_{\ep,C}(1)$ that contains all but $n'$ elements of $A$ (counting multiplicity), where 

$$|P|=O_{C,\ep}(\rho^{-1}/{n'}^{r/2}).$$ 
\end{theorem}

The Erd\H{o}s and Littlewood-Offord inequality emerges in the singularity  problem of random matrices by the following simple observation: condition on the matrix $A_{n-1}$ of the first $n-1$ rows, the determinant of $A_n$ can be expressed as a linear form of its last row $\Bx=(x_1,\dots, x_n)$, 

\begin{equation}\label{eqn:det}
\det(A_n)=x_1a_1+\dots+x_na_n,
\end{equation}

where $a_1,\dots,a_n$ are the corresponding cofactors of the matrix $A_{n-1}$. Thus, if we can prove that most of the $a_i$ are non-zero (with respect to $A_{n-1}$), then it follows from the Erd\H{o}s and Littlewood-Offord inequality that $\P_{\Bx}(\det(A_n)=0)=O(n^{-1/2})$. This approach was initiated by Koml\'os. 

To prove Theorem \ref{theorem:main}, we reverse the approach of Koml\'os. First we express $\det(Q_n)$ as a linear form if its last row $\Bx=(x_1,\dots,x_n)$ 

$$\det(Q_n)=x_1a_1+\dots+x_na_n,$$

where $a_1,\dots,a_{n}$ are the cofactors of $Q_{n-1}$.

Roughly speaking, our method consists of two main steps.

\begin{itemize}
\item (Inverse Step) Assume that 

$$\rho^\ast(A):=\sup_a\P_\Bx(x_1a_1+\dots+a_{n}x_n=a)\ge n^{-C},$$ 

\noindent where the probability is taken uniformly over all $(0,1)$ tuples $(x_1,\dots,x_n)$ for which there are exactly $n/2$ zero components, then there is a strong structure among the cofactors $a_1,\dots,a_n$.
\vskip .2in
\item (Counting Step) With respect to $Q_{n-1}$, the event that there exists a strong structure among the $a_i$'s happens with negligible probability. 
\end{itemize}

For the rest of this section, we focus on the Inverse Step. Details of the Counting Step will be presented in Section \ref{section:proof}. 

First, by using the relation $\rho(A)=\Omega(\rho^{\ast}(A)/\sqrt{n})$, we can deduce directly from Theorem \ref{theorem:ILO} the following result.

\begin{theorem}[Inverse Littlewood-Offord result for $\rho^\ast(A)$, I]\label{theorem:ILOa}
Let $\eps<1$ and $C$ be positive constants. Assume that

$$\rho^\ast(A)  \ge  n^{-C}. $$

Then, for any $n^\ep \le n' \le n$, there exists a proper symmetric GAP $P$ of rank $r=O_{\ep,C}(1)$ that contains all but $n'$ elements of $A$ (counting multiplicity), where 

$$|P|=O_{C,\ep}\big((\rho^\ast(A))^{-1}\sqrt{n}/{n'}^{r/2}\big).$$ 
\end{theorem}

%In some cases, the bound $\rho^{\ast}(A)^{-1}\sqrt{n}/{n'}^{r/2}$ in Theorem \ref{theorem:ILOa} is not quite useful. For instance if we let $A$ be the set $\{1,2,\dots,n\}$. Then it follows from heuristic that $\rho^\ast(A)=\Theta(n^{-3/2})$ . Thus, by applying Theorem \ref{theorem:ILOa} with $n'=\ep n$ to minimize the GAP's size, we obtain a GAP $P$ which contains almost all elements of $A$ and which has size $|P|=O(\rho^{\ast}(A)^{-1}\sqrt{n}/(\ep n)^{r/2})=O(n^2/n^{r/2})$. This means that there might exist a GAP of rank $r=1$ and size $O(n^{3/2})$ that contains most of the elements of $A$. However, this bound is not useful because there clearly exist symmetric proper GAPs of rank 1 and size $\Theta(n)$ that contain $A$.

Observe that if $A'$ is a translation of $A$, $A'=\{a_1+N,\dots,a_n+N\}$ for some $N$, then $\rho^\ast(A)=\rho^\ast(A')$. On the other hand, symmetric GAPs do not preserve under translation. This suggests that we might obtain a more economical version of Theorem \ref{theorem:ILOa} by passing to consider non-symmetric GAPs. 

\begin{theorem}[Inverse Littlewood-Offord result for $\rho^\ast$, II]\label{theorem:ILOb}
Let $\eps<1$ and $C$ be positive constants. Assume that $n^{\ep}\le n' < n$ and $A=\{a_1,\dots,a_n\}$ is a multiset for which there are no more than $n-n'-1$ elements taking the same value. Assume furthermore that

$$\rho^\ast(A)  \ge  n^{-C}. $$

Then, there exists a (not necessarily symmetric) proper GAP $P$ of rank $2\le r=O_{\ep,C}(1)$ that contains all but $n'$ elements of $A$ (counting multiplicity), where 

$$|P|=O_{C,\ep}\big((\rho^\ast(A))^{-1}\sqrt{n}/{n'}^{r/2}\big).$$ 

In particular,

$$|P|=O_{C,\ep}\big((\rho^\ast(A))^{-1}\sqrt{n}/n'\big).$$ 

\end{theorem}

Remark that the essential advantage of Theorem \ref{theorem:ILOa} over Theorem \ref{theorem:ILOb} is that the rank $r$ must be at least 2, which leads to a ''gain'' of a factor $\sqrt{n'}$ in the size of $|P|$.

\begin{proof}(of Theorem \ref{theorem:ILOb}) We consider a translation $A'=\{a_1',\dots,a_n'\}$ of $A$, where $a_i':=a_i+N$ and $N$ is chosen to be sufficiently large. It is clear that 
$$\rho^\ast(A')=\rho^\ast(A).$$

Consider the concentration probability of $A'$. Because 

$$\rho(A')=\Omega(\rho^\ast(A')/\sqrt{n})=\Omega(\rho^\ast(A)/\sqrt{n})= n^{-O(1)},$$ 

Theorem \ref{theorem:ILO} implies that there exists a symmetric GAP $Q$ of rank $r=O(1)$ and size $(\rho(A'))^{-1}/n'^{r/2}$ that contains all but $n'$ elements of $A'$. 

We now show that the rank $r$ of $Q$ must be at least 2. Assume otherwise that $Q$ has rank one. Let $d$ and $L$ be its step and length respectively. Because $Q$ contains at least two distinct elements $a_{i_1}$ and $a_{i_2}$, so 

$$d\le |a_{i_1}-a_{i_2}|\le 2\sum_i |a_i|.$$ 

On the other hand, we also have $Ld\ge |N+a_1|\ge N-\sum_i|a_i|$, and so 

$$d\ge (N-\sum_i|a_i|)/L = \Omega\big((N-\sum_i|a_i|)n^{-O(1)}\big).$$ 

This bound contradicts with the upper bound  $2\sum_i |a_i|$ if $N$ was chosen to be large enough, and so $Q$ must have rank at least 2.

Next, recall that $\rho(A')=\Omega(\rho^\ast(A')/\sqrt{n})=\Omega(\rho^\ast(A)/\sqrt{n})$. Thus,  

$$|Q|=O\big((\rho^\ast(A))^{-1}\sqrt{n}/n'^{r/2}\big).$$

To complete the proof, we shift $Q$ by $-N$ to obtain a GAP $P$ of the same rank and size which contains all but $n'$ elements of $A$. 
\end{proof}

We now draw two quick consequences of Theorem \ref{theorem:ILOb}. 

Our first result, which is similar to the inequality of Erd\H{o}s and Littlewood-Offord, asserts that as long as $A$ is not too trivial, $\rho^\ast(A)$ is small.

\begin{corollary}[Erd\H{o}s-Littlewood-Offord inequality for $\rho^\ast$]\label{cor:forward}
Let $\eps<1$ be a positive constant. Assume that $n^{1/2+\ep}\le n' \le n$ and $A=\{a_1,\dots,a_n\}$ is a multiset where there are no more than $n-n'-1$ elements taking the same value. Then we have 

$$\rho^\ast(A)=O_\ep(\sqrt{n}/n').$$ 
\end{corollary}

\begin{proof}(of Corollary \ref{cor:forward})
Assume that for some sufficiently large $C$

$$\rho^\ast(A)\ge C \sqrt{n}/n'.$$ 

Then, by Theorem \ref{theorem:ILOb}, there exists a GAP $P$ of size $|P|=O_\ep((\rho^\ast))^{-1}\sqrt{n}/n')$ which contains all but $n'$ elements of $a_i$. Because there are no more than $n-n'-1$ elements among $a_i$ taking the same value, $P$ must have size at least 2. On the other hand, the upper bound $|P|=O_\ep((\rho^\ast))^{-1}\sqrt{n}/n')=O_\ep(C^{-1})$ shows that $P$ can be forced to be empty if we choose $C$ to be sufficiently large depending on $\ep$. This is a contradiction, and so 

$$\rho^\ast(A)= O_\ep(\sqrt{n}/n').$$ 
\end{proof}

Improving earlier result of Erd\H{o}s and Moser \cite{EM}, S\'ak\"ozy and Szemer\'edi proved in \cite{SSz}  that if $a_i$ are distinct then $\rho(A)=O(n^{-3/2})$. Our next consequence shows a similar bound for $\rho^\ast(A)$. 

\begin{corollary}[S\'ark\"ozy-Szemer\'edi theorem for $\rho^\ast$]\label{cor:distinct}
Assume that $A=\{a_1,\dots,a_n\}$, where $a_i$ are distinct real numbers. Then we have

$$\rho^\ast(A):=O(n^{-3/2}).$$
\end{corollary}

\begin{proof}(of Corollary \ref{cor:distinct})
Assume that for some sufficiently large $C$

$$\rho^\ast(A)\ge Cn^{-3/2}.$$ 

Then, Theorem \ref{theorem:ILOb} applied to $n'=\ep n$ implies that there exists a GAP $P$ of size $|P|=O_\ep((\rho^\ast))^{-1}\sqrt{n}/n')$ which contains all but $\ep n$ elements of $a_i$. Because the elements of $A$ are distinct, $P$ must have size at least $(1-\ep)n$. On the other hand, the upper bound $|P|=O_\ep((\rho^\ast))^{-1}\sqrt{n}/\ep n)=O_\ep(C^{-1}n)$ shows that $P$ has size much smaller than $\ep n$ once we choose $C$ to be sufficiently large depending on $\ep$. This is a contradiction, and so

$$\rho^\ast(A)=O(n^{-3/2}).$$ 
\end{proof}

\section{Further supporting lemmas}\label{section:lemmas}

\subsection{A rank reduction argument and the full rank assumption} This section provides a technical lemma we will need for later sections. Informally, it says that if we can find a proper GAP that contains a given set, then we can assume this containment is non-degenerate. More details are followed.  

Assume that $P=\{g_0+m_1g_1+\dots+m_rg_r | N_i\le m_i \le N_i'\}$ is a proper GAP, which contains a set $U=\{u_1,\dots. u_n\}$. 

We consider $P$ together with the map $\Phi: P \rightarrow \R^r$ which maps $g_0+m_1g_1+\dots+m_rg_r$ to $(m_1,\dots,m_r)$. Because $P$ is proper, this map is bijective. 

We know that $P$ contains $U$, but we do not know yet that $U$ is non-degenerate in $P$ in the sense that the set $\Phi(U)$ has full rank in $\R^{r}$. In the later case, we say $U$ {\it spans} P.

\begin{lemma}\label{lemma:fullrank}
Assume that $U$ is a subset of a proper GAP $P$ of size $r$, then there exists a proper  GAP $Q$ that contains $U$ such that the followings hold.

\begin{itemize}
\item $\rank(Q)\le r$ and $|Q|\le O_r(1)|P|$;

\vskip .1in

\item $U$ spans $Q$, that is, $\phi(U)$ has full rank in $\R^{\rank(Q)}$.
\end{itemize}

\end{lemma}

To prove Theorem \ref{lemma:fullrank}, we will rely on the following lemma.

\begin{lemma}[Progressions lie inside proper progressions]\cite[Theorem 3.40]{TVbook}\label{lemma:embeding}
There is an absolute constant $C$ such that the following holds. Let $P$ be a GAP of rank $r$ in $\R$. Then there is a proper GAP $Q$ of rank at most $r$ containng $P$ and 

$$|Q|\le r^{Cr^3}|P|.$$ 

\end{lemma}

\begin{proof} (of Lemma \ref{lemma:fullrank}) We shall mainly follow \cite[Section 8]{TVsing} and \cite[Section 6]{Ng-Duke}.

Suppose that $\Phi(U)$ does not have full rank, then it is contained in a hyperplane of $\R^r$. In other words, there exist integers $\alpha_1,\dots,\alpha_r$ whose common divisor is one and $\alpha_1m_1+\dots + \alpha_r m_r=0$ for all $(m_1,\dots,m_r)\in \Phi(U)$.

Without loss of generality, we assume that $\alpha_r \neq 0$. We select $w$ so that $g_r=\alpha_r w$, and consider $P'$ be the GAP generated by $g_i':=g_i-\alpha_iw$ for $1\le i \le r-1$. The new GAP $P'$ will continue to contain $U$, because we have 

\begin{align*}
g_0+m_1g_1'+\dots +m_{r-1}g_{r-1}' &=g_0+ m_1g_1+\dots+m_rg_r - w(\alpha_1m_1+\dots+\alpha_rg_r)\\
&= m_1g_1+\dots+m_rg_r
\end{align*}

for all $(m_1,\dots,m_r)\in \Phi(U)$. 

Also, note that the volume of $P'$ is $2^{r-1}(N_1'-N_1+1)\cdots (N_{r-1}'-N_{r-1}+1)$, which is less than the volume of $P$.

We next use Lemma \ref{lemma:embeding} to guarantee that $P'$ is proper without increasing the rank. 
 
Iterate the process if needed. Because we obtain a new proper GAP whose rank strictly decreases each step, the process must terminate after at most $r$ steps.
 
\end{proof}

\subsection{Rational commensurability} This section shows that if $\Bv$ is a normal vector of a hyperplane generated by vectors of bounded integral components, and if most of the components of $\Bv$ belong to a GAP, then {\it all} of the components belong to a GAP of rank 1 and of relatively small length. 

\begin{lemma}[Rational commensurability]\label{lemma:rational}
Let $v=(v_1,\dots, v_{m})$ be a vector such that all but $m^\ep$ components $v_i$ belong to a proper symmetric GAP of rank $O_{C,\ep}(1)$ and size $m^{O_{C,\ep}(1)}$, and that $v$ is a normal vector of a hyperplane spanned by vectors of integral components bounded by $m^{O_{C,\ep}(1)}$. Then $\{v_1,\dots,v_{m}\} \subset \{(p/q)v_{i_0}, |p|,|q| \le m^{O_{C,\ep}(m^\ep)}\}$ for some $i_0$. \end{lemma}

This is basically \cite[Lemma 9.1]{Ng-Duke}. Allow us to reprove it where for the completeness of the note. 

\begin{proof}(of Lemma \ref{lemma:rational})
Without loss of generality, we assume that $(v_{m-m^\ep},\dots, v_{m})$ are the exceptional elements that may not belong to the GAP. 

For each $v_i$, where $i< m-m^\ep$, there exist numbers $v_{ij}$ bounded by $m^{O_{C,\ep}(1)}$ such that 

$$v_i = v_{i1}g_1+\dots v_{ir}g_r,$$ 

where $g_1,\dots, g_r$ are the generators of the GAP.

Consider the $m$ by $r+m^\ep$ matrix $M_v$ whose $i$-th column is the vector $(v_{i1},\dots,v_{ir},0,\dots,0)$ if $i< m-m^\ep$, and $(0,\dots,0,1,0,\dots,0)$ if $m-m^\ep\le i$. Note that $M_v$ has rank $r+m^\ep$. 

We thus have  

$$v^T=M_v \cdot u^T,$$

where $u=(g_1,\dots,g_r,v_{m-m^\ep},\dots,v_{m})$.

Next, let $w_1,\dots,w_{m-1}$ be the vectors of integral entries bounded by $m^{O_{C,\ep}(1)}$ which are orthogonal to $v$. We form a $m$ by $m$ matrix $M_w$ whose $i$-th row is $w_i$ for $i\le n-2$, and the $m$-th row is $e_{i_0}$, the unit vector among the standard basis $\{e_1,\dots, e_{m}\}$ that is linearly independent to $w_1,\dots, w_{m-1}$. 

By definition, we have $M_w v^T = (0,\dots,0,v_{i_0},0,\dots,0)^T$, and hence 

$$(M_wM_v) u^T =(0,\dots,0,v_{i_0},0,\dots,0)^T.$$ 

The indentity above implies that

\begin{equation}\label{eqn:smallvector:1}
(M_{w}M_{v}) (\frac{1}{v_{i_0}} \cdot u)^T =(0,\dots,0,1,0,\dots,0)^T.
\end{equation}

Next we choose a submatrix $M$ of size $r+m^\ep$ by $r+m^\ep$ of $M_{w}M_{v}$ thas has full rank. Then

\begin{equation}\label{eqn:smallvector:2}
M  (\frac{1}{v_{i_0}}\cdot u)^T = x
\end{equation}

for some $x$ which a subvector of $(0,\dots,0,1,0,\dots,0)$ from \eqref{eqn:smallvector:1}.

Observe that the entries of $M$ are integers bounded by $m^{O_{C,\ep}(1)}$. Solving for $g_i/v_{i_0}$ and $v_{j}/v_{i_0}$ from \eqref{eqn:smallvector:2}, we conclude that each component can be written in the form $p/q$, where $|p|,|q| \le m^{O_{C,\ep}(m^\ep)}$. 

\end{proof}

\section{Proof of Theorem \ref{theorem:main}}\label{section:proof}

In this section we prove our main theorem. Let $\Br_1,\dots,\Br_n$ be the row vectors of $Q_n$, we will show 

\begin{theorem}\label{theorem:m}
Let $1\le m\le n-1$. Assume that $Q_n$ is a random $(0,1)$ matrix whose rows are independent vectors of exactly $n/2$ zero components. Then for any $C>0$
$$\P(\Br_1,\dots,\Br_m \mbox{ generate a subspace } \BH \mbox{ of dimension } m, \mbox { and } \Br_{m+1}\in \BH)=O(n^{-C}),$$

where the implied constant depends on $C$.
\end{theorem}

It is clear that Theorem \ref{theorem:m} implies Theorem \ref{theorem:main}. We next claim that its suffices to work with the case of $m$ being large, $m=n-O(\log n)$.

\begin{claim}\label{claim:2}
Let $\BH\subset \R^n$ be a subspace of dimension $m=n-\omega(\log n)$, and $\Br=(r_1,\dots,r_n)$ be a random $(0,1)$ vector of exactly $n/2$ zero components. Then we have 

$$\P(\Br\in \BH) =n^{-\omega(1)}.$$
\end{claim}

\begin{proof}(of Claim \ref{claim:2})
Because $\Br\in \BH$, there exist $m$ components, say $r_{i1},\dots,r_{im}$, such that for any $d \notin\{i_1,\dots,i_m\}$, there exist numbers $\alpha_{d1},\dots,\alpha_{dm}$ depending on $\BH$ such that 

$$r_d=\sum_{j=1}^m \alpha_{dj}r_{i_j}.$$ 

Thus we can bound the probability $\P(\Br\in \BH)$ by

$$\P(\Br\in \BH)\le \sum_{\max(0,m-n/2)\le t\le \min(m,n/2)}\binom{m}{t}/\binom{n}{n/2} \le 2^{m}/\binom{n}{n/2}.$$

As $m=n-\omega(\log n)$, it is clear that $2^m= 2^n/n^{\omega(1)}$,  and so 

$$\P(\Br\in \BH) =n^{-\omega(1)}.$$

\end{proof}

Because $\BH$, the subspace generated by $\Br_1,\dots,\Br_m$, has rank $m$, there is a nonzero vector $\Ba=(a_1,\dots,a_{m+1},0,\dots,0)$ which is orthogonal to $H$ and $a_i=0$ for $m+2\le i\le n$. Thus, condition on $\Br_1,\dots,\Br_m$, the probability that the next row $\Br_{m+1}=(x_1,\dots,x_n)$ belongs to $\BH$ is bounded by $\rho^\ast(A)=\P(x_1a_1+\dots+x_na_n=0|\Br_1,\dots,\Br_m)$, where $A=\{a_1,\dots,a_n\}$. 

For Theorem \ref{theorem:m}, there is nothing to prove if $\rho^\ast(A)\le n^{-C}$. Assume otherwise, we will apply the following structural result. 

\begin{lemma}[Inverse Step]\label{lemma:m}
Assume that  

$$\rho^\ast(A)=\P(x_1a_1+\dots+x_na_n=0|\Br_1,\dots,\Br_m)\ge n^{-C}.$$

Then there exists a nonzero vector $\Bu=(u_1,\dots,u_n)$ which satisfies the following properties.

\begin{itemize}
\item $\Bu$ is orthogonal to the rows $\Br_1,\dots,\Br_m$, and $\rho^\ast(U)=\rho^\ast(A)$, where $U=\{u_1,\dots,u_n\}$.
\vskip .1in
\item Every $u_i$ can be written in the form $p/q$, where $p,q$ are integers and $|p|,|q| \le n^{O_{C,\ep}(n^\ep)}$.
\vskip .1in
\item Either there are $n-n^{1/2+\ep}$ components $u_i$ that take the same value or all but $n^{1/2+\ep}$ of $u_i$ (counting multiplicity) belong to a GAP of rank $O_{C,\ep}(1)$ and of size $O_{C,\ep}\big((\rho^\ast(U))^{-1}/n^{\ep}\big)$, and $\rho^\ast(U)=O_{\ep}(n^{-\ep})$.
\end{itemize}
\end{lemma}

Assume Lemma \ref{lemma:m} for the moment, we next proceed to the Counting Step to complete the proof of the main result of this section.

\begin{proof}(of Theorem \ref{theorem:m}) We will consider two cases, depending on the structure of $\Bu$.

{\bf Case 1}. We first consider the probability $\P^\sharp$ of the event that the first $m$ rows of $Q_{n-1}$ are orthogonal to a nonzero vector $\Bu$ for which there are $n_0:=n-n^{1/2+\ep}$ components $u_i$ taking the same value. 

By paying a factor $\binom{n}{n_0}$ in probability, we may assume that $u_1=\dots=u_{n_0}$. Thus 

$$\Bu=(u_1,\dots,u_1,u_{n_0+1},\dots,u_n).$$ 

Let $Q_m$ be the matrix of the first $m$ rows of $Q_n$, and let $Q$ be the $m$ by $n-n_0$ matrix whose first column is the sum of $Q_{m}$'s first $n_0$ columns, and $Q$'s $i$-th column is $Q_{m}$'s $(n_0+i)$-th column for other $i$. By definition, the row vectors of $Q$ are orthogonal to the nonzero vector $\Bu'=(u_1,u_{n_0+1},\dots,u_n)$. Thus $Q$ has rank at most $n-n_0-1$. 

We now bound the probability of the event that $Q$ has rank at most $n-n_0-1$. By paying another factor of $\binom{n-1}{n-n_0-1}$ in probability, we may assume that any row of index at least $n-n_0$ belongs the subspace $\BH$ generated by the first $n-n_0-1$ rows of $Q$. We will rely on the following simple claim.

\begin{claim}\label{claim:1} Let $\ep<1/4$ be a fixed constant. Let $\BH$ be a hyperplane in $\R^{n-n_0}$, and $\Bx=(x_1,\dots,x_{n-n_0})$ be a random vector such that 

$$\P(x_1=k)=\binom{n_0}{k}\binom{n-n_0}{n/2-k}/\binom{n}{n/2} \mbox{ for } 0\le k\le n/2,$$ 

and $(x_2,\dots,x_{n-n_0})$ are chosen uniformly from all $(0,1)$ vectors of exactly $n/2-x_1$ unit components. Then, as $n$ is sufficiently large, one has 

$$\P(\Bx\in \BH)\le 3/4+\ep.$$   
\end{claim}

\begin{proof}(of Claim \ref{claim:1})
Let $\Bh=(h_1,\dots,h_{n-n_0})$ be the normal vector of $\BH$. We first assume that there exists $2\le i_0<j_0\le n-n_0$ such that $h_{i_0}\neq h_{j_0}$. Without loss of generality, assume furthermore that $i_0=n-n_0-1$ and $j_0=n-n_0$. It then follows that for any chosen tuple $(f_1,\dots,f_{n-n_0-2})$, either $(f_1,\dots,f_{n-n_0-2},0,1)$ or $(f_1,\dots,f_{n-n_0-2},1,0)$ does not belong to $\BH$. On the other hand, it is follows from the distribution of $\Bx$ that the probability $\P_\Bx(x_1+\dots+x_{n-n_0-2}=n/2-1)$ is at least $1/2-2\ep$. So we have 

$$\P(\Bx\in \BH)\le 3/4+\ep.$$

Now assume that $h_2=\dots=h_{n-n_0}=h$. In this case, $\Bx \in \BH$ if and only if $x_1h_1+h(n/2-x_1)=0$. On the other hand, by the distribution of $x_1$, it is clear that $\P(x_1=k)\le 3/4$ for any $k$, so one also has $\P(\Bx\in \BH)<3/4$ in this case.
\end{proof}

By Claim \ref{claim:1}, the probability that all rows of $Q$ of index at least $n-n_0$ belong the subspace $\BH$ generated by $Q$'s first $n-n_0-1$ rows is bounded by 

$$(3/4+\ep)^{n_0}.$$

Putting everything together, we obtain the following bound for $\P^\sharp$

$$\P^\sharp \le \binom{n}{n_0}\binom{n-1}{n-n_0-1}(3/4+\ep)^{n_0}= (3/4+\ep)^{(1-o(1))n}.$$

{\bf Case 2}. We consider the probability $\P^\flat$ of the event that the first $m$ rows of $Q_{n}$ are orthogonal to a nonzero vector $\Bu$ for which the following properties hold.

\begin{itemize}
\item $n^{-C}\le \rho^\ast(U)=O_\ep(n^{-\ep})$.
\vskip .1in
\item Every $u_i$ can be written in the form $p/q$, where $|p|,|q| \le n^{O_{C,\ep}(n^\ep)}$.
\vskip .1in
\item All but $n^{1/2+\ep}$ of the $u_i$'s belong to a GAP of rank $O_{C,\ep}(1)$ and of size $O\big((\rho^\ast(U))^{-1}/n^{\ep}\big)$.
\end{itemize}

Let $0<\delta$ to be chosen. We divide the interval $[n^{-C},O_\ep(n^{-\ep})]$ into sub-intervals $[n^{-(k+1)\delta},n^{-k\delta}]$, where $\ep/\delta \le k\le C/\delta$.  For each $k$, let $\BG_k$ be the collection of $\Bu$'s such that $\rho^\ast(U)\in [n^{-(k+1)\delta},n^{-k\delta}]$, and let $\P_k$ be the probability that the first $m$ rows of $Q_{n-1}$ are orthogonal to one of $\Bu$ from $G_k$. 

We now bound the size of $\BG_k$. To do this, we first count the number of GAPs which may contain most of the components of vectors $\Bu$ from $\BG_k$, and then count the number of $\Bu$'s whose components are chosen from the given structures. By Lemma \ref{lemma:fullrank}, one can assume that all the GAP generators are of the form $p/q$, where $|p|,|q|\le (n^{O_{C,\ep}(n^\ep)})^{O_{C,\ep}(1)}=n^{O_{C,\ep}(n^\ep)}$. Because each GAP has rank $O_{C,\ep}(1)$ and size $O((\rho^\ast)^{-1}/n^\ep)=O(n^{\delta(k+1)}/n^\ep)$, the number of such GAPs is bounded by 

$$(n^{O_{C,\ep}(n^\ep)})^{O_{C,\ep}(1)} (n^{\delta(k+1)}/n^\ep)^{O_{C,\ep}(1)}= O(n^{O_{C,\ep}(n^\ep)}).$$

After choosing a GAP of size $O(n^{\delta(k+1)}/n^\ep)$, the number of ways to choose $n-n^{1/2+\ep}$ of the $u_i$'s as its elements is

$$\binom{n}{n^{1/2+\ep}}\binom{O(n^{\delta(k+1)}/n^\ep)}{n-n^{1/2+\ep}}=O\big(n^{n^{1/2+\ep}}(n^{\delta(k+1)}/n^\ep)^{n-n^{1/2+\ep}}\big).$$

For the remaining $n^{1/2+\ep}$ exceptional elements, there are $(n^{O_{C,\ep}(n^\ep)})^{n^{1/2+\ep}}=n^{O_{C,\ep}(n^{1/2+2\ep})}$ ways to choose them from the set  $\{p/q, |p|,|q|\le n^{O_{C,\ep}(n^\ep)})\}$. 

Putting these bounds together, we obtain the following bound for the number of $\Bu$ of $\BG_k$

$$|\BG_k|=O\big(n^{O_{C,\ep}(n^{1/2+2\ep})}(n^{\delta(k+1)}/n^\ep)^{n-n^{1/2+\ep}}\big).$$

Now, for given $\Bu\in \BG_k$, the probability that the first $m$ rows of $Q_{n-1}$ are orthogonal to $\Bu$ is $(\rho^\ast(U))^m\le (n^{-\delta k})^m$. Thus we can estimate $\P_k$ as 

$$\P_k \le |\BG_k|(n^{-\delta k})^{m} =O\big(n^{O_{C,\ep}(n^{1/2+2\ep})} (n^\delta)^{n}/(n^\ep)^{n-n^{1/2+\ep}}\big)=o(n^{-\ep n/2}),$$

provided that $\delta$ was chosen to be smaller than $\ep/3$.

Summing over $k$, we thus obtain

$$\P^\flat=\sum_{k\le C/\delta} \P_k= o(n^{-\ep n/2}).$$

\end{proof}

It remains to justify the Inverse Step.

\begin{proof}(of Lemma \ref{lemma:m}) Let $A$ denote the multiset $\{a_1,\dots,a_n\}$. As $\rho^\ast(A) \ge n^{-C}$, we have  

$$\rho(A)=\Omega(\rho^\ast(A)/\sqrt{n})= \Omega(n^{-C+1/2}).$$ 

We first apply Theorem \ref{theorem:ILO} to the multiset $A=\{a_1,\dots,a_n\}$ with $n'=n^\ep$ to obtain a GAP of rank $O_{C,\ep}(1)$ and size $n^{O_{C,\ep}(1)}$ that contains all but $n^\ep$ elements of $A$. Next, because $\Ba=(a_1,\dots,a_{m+1})$ is orthogonal to the rows of $Q_{m}$, the matrix of rank $m$ generated by the first $m$ rows of $Q$, we infer from Lemma \ref{lemma:rational} that all of the $a_i$'s have the form $(p/q)\cdot a_{i_0}$ with some integers $|p|,|q| \le n^{O_{C,\ep}(n^\ep)}$ and with some fixed $i_0$.

Set 

$$\Bu=(u_1,\dots,u_n):=\frac{1}{a_{i_0}}\cdot \Ba.$$ 

It is clear that the nonzero vector $\Bu$ is orthogonal to $\Br_1,\dots,\Br_m$, and $\rho^\ast(\{u_1,\dots,u_n\})=\rho^\ast(A)$. We next consider two cases. 

{\bf Case 1.} If there are more than $n-n^{1/2+\ep}$ components $u_i$ which take the same value, then there is nothing to prove.

{\bf Case 2.} If this is not the case, then by Corollary \ref{cor:forward} we have $\rho^\ast(U)=O(n^{-\ep})$. Next, because $\rho^\ast(U)\ge n^{-C}$, we apply Theorem \ref{theorem:ILOb} to the multiset $U=\{u_1,\dots,u_n\}$ with $n'=n^{1/2+\ep}$ to obtain a GAP of rank $2\le r=O_{C,\ep}(1)$ and size 

$$|P|=O_{C,\ep}(\rho^\ast \sqrt{n}/n')=O_{C,\ep}\big((\rho^\ast(U))^{-1}/n^{\ep}\big).$$ 

\end{proof}

\end{document}